\documentclass[11pt,reqno]{amsart}
\usepackage{amsthm}
\usepackage{amsmath,amssymb,graphicx,amscd,xy}
\usepackage[usenames,dvipsnames]{color}
\usepackage{graphicx}
\usepackage{longtable}
\usepackage{mathtools}
\usepackage{color}
\usepackage{verbatim}
\usepackage{pict2e}
\usepackage[utf8,utf8x]{inputenc}
\usepackage{enumitem}
\usepackage{longtable}
\usepackage{hyperref}

\parskip=\smallskipamount

\begin{document}

\author{Lars Simon}
\address{Lars Simon, Department of Mathematical Sciences, Norwegian University of Science and Technology, Trondheim, Norway}
\email{lars.simon@ntnu.no}

\thanks{Part of this work was done during the international research program "Several Complex Variables and Complex Dynamics" at the Centre for Advanced Study at the Academy of Science and Letters in Oslo during the academic year 2016/2017.}
\title{A Homogeneous Function Constant Along The Leaves Of A Foliation}


\begin{abstract}
Given a smooth foliation by complex curves (locally around a point $x\in\mathbb{C}^2\setminus\{0\}$) which is ``compatible'' with the foliation by spheres centered at the origin, we construct a smooth real-valued function $g$ in a neighborhood of said point, which is positive, homogeneous and constant along the leaves. A corollary we obtain from this is relevant to the problem of ``bumping out'' certain pseudoconvex domains in $\mathbb{C}^3$.
\end{abstract}

\maketitle

\section{Introduction}
The technique of ``bumping out'' bounded, smoothly bounded pseudoconvex domains of finite D'Angelo $1$-type in $\mathbb{C}^{n+1}$, $n\geq{1}$, has proven to be useful both in the construction of peak functions (see e.g.\ \cite{MR0492400}, \cite{MR1016439}) and in the construction of integral kernels for solving the $\overline{\partial}$-equation (see e.g.\ \cite{MR835766}, \cite{MR1070924}).

As in \cite{MR2452636}, a local bumping of a smoothly bounded pseudoconvex domain $\Omega\subseteq\mathbb{C}^{n+1}$, $n\geq{1}$, at a boundary point $\zeta\in\partial\Omega$ is defined to be a triple $(\partial\Omega{},U_{\zeta},\rho_{\zeta})$, such that:
\begin{itemize}
\item{$U_{\zeta}\subseteq\mathbb{C}^{n+1}$ is an open neighborhood or $\zeta$,}
\item{$\rho_{\zeta}\colon{}U_{\zeta}\to\mathbb{R}$ is smooth and plurisubharmonic,}
\item{$\rho_{\zeta}^{-1}(\{0\})$ is a smooth hypersurface in $U_{\zeta}$ that is pseudoconvex from the side $U_{\zeta}^{-}:=\{z\colon{}\rho_{\zeta}(z)<0\}$,}
\item{$\rho_{\zeta}(\zeta)=0$, but $\rho_{\zeta}<0$ on $U_{\zeta}\cap\left(\overline{\Omega}\setminus{\{\zeta\}}\right)$.}
\end{itemize}
A priori, such a local bumping needs to have additional properties for the upper mentioned constructions to work; specifically, when assuming $\Omega$ to be of finite type (``type'' refers to the D'Angelo $1$-type), one desires the order of contact between $\partial\Omega$ and $\rho_{\zeta}^{-1}(\{0\})$ at $\zeta$ to not exceed the type of $\zeta$ in any direction.

As seen in, e.g., \cite{MR2452636}, attempts to construct such a local bumping with the desired additional properties naturally lead to the problem of bumping homogeneous plurisubharmonic polynomials on $\mathbb{C}^n$. While it is not obvious how bumping results for homogeneous plurisubharmonic polynomials can be used to obtain {\emph{useful}} bumping results for the domains that motivate their study, Noell \cite{MR1207878} and Bharali \cite{MR2993440} have been successful in doing so.\\
Hence, bumping results for homogeneous plurisubharmonic polynomials are an important first step towards obtaining useful bumping results for domains.

Specifically, assume we are given a real-valued polynomial $P\not\equiv{0}$ with complex coefficients in $n$ complex variables $z_1,\dots{},z_n$ and their conjugates $\overline{z_1},\dots{},\overline{z_n}$. Furthermore assume that
\begin{itemize}
\item{$P$ is $\mathbb{R}$-homogeneous of degree $2k$, for some positive integer $k\geq{2}$,}
\item{$P$ is plurisubharmonic,}
\item{$P$ does not have any pluriharmonic terms.}
\end{itemize}

In this setting the question becomes, roughly speaking, how much one can subtract from $P$ without destroying plurisubharmonicity and while preserving homogeneity.\\
If $P$ is additionally assumed to not be harmonic along any complex line through $0\in\mathbb{C}^n$, then there exists a smooth function $F\colon{\mathbb{C}^n\setminus\{0\}}\to\mathbb{R}$, such that $F$ is positive, $\mathbb{R}$-homogeneous of degree $2k$ and such that $P-F$ is strictly plurisubharmonic on $\mathbb{C}^n\setminus\{0\}$ (the assumption that $P$ does not have any pluriharmonic terms is clearly not necessary for this). In the case $n=1$ this follows from a stronger result by Forn{\ae}ss and Sibony \cite[Lemma 2.4]{MR1016439}. In the case $n\geq{2}$ this was shown by Noell \cite{MR1207878}. Since the case $n=1$ is completely solved by this, we will assume $n\geq{2}$ from now on.

If, however, $P$ is allowed to be harmonic along complex lines through $0$, then one cannot expect to obtain such a strong result. The next best bumping result one could hope for is the existence of a function $H\colon\mathbb{C}^n\to\mathbb{R}$ having the following properties:
\begin{itemize}
\item{$H$ is $\mathbb{R}$-homogeneous of degree $2k$ and smooth away from 0,}
\item{$H\geq{0}$ everywhere with equality precisely in $0$ and along all complex lines through $0$ along which $P$ is harmonic (i.e.\ vanishes, since $P$ does not have any pluriharmonic terms),}
\item{$P-H$ is plurisubharmonic,}
\item{$\left|{\frac{P}{H}}\right|$ is bounded on ${\mathbb{C}}^n\setminus{}H^{-1}(\{0\})$.}
\end{itemize}
In dimension $n=2$, Bharali and Stens{\o}nes \cite{MR2452636} have obtained such bumping results in two cases, which, in some sense, can be interpreted as the two ``extremal behaviors'' the Levi-degeneracy set of $P$ can exhibit, when $P$ is allowed to be harmonic along complex lines through $0$.

In one of the upper mentioned cases studied by Bharali and Stens{\o}nes \cite{MR2452636}, the polynomial $P\colon\mathbb{C}^2\to\mathbb{R}$ is assumed to be harmonic along the smooth part of every level set of a non-constant entire function. They proceed by showing that $P$ can be written as the composition of a subharmonic homogeneous polynomial on $\mathbb{C}$ with a holomorphic homogeneous polynomial on $\mathbb{C}^2$. The bumping for $P$ is then constructed by applying the result by Forn{\ae}ss and Sibony \cite[Lemma 2.4]{MR1016439}. So, roughly speaking, they identify a foliation by complex curves along which $P$ is harmonic and then bump with something that is homogeneous of degree $2k$, constant along the leaves of the foliation and positive away from a small singular set.

If this assumption, that such an entire function exists, is replaced by the {\emph{weaker}} assumption that the determinant of the Complex Hessian Matrix of $P$ vanishes identically on $\mathbb{C}^2$, then no bumping results for $P$ are known thus far. Applying the Frobenius theorem, however, one {\emph{does}} obtain a foliation as above, albeit not necessarily a holomorphic one (see also the paper by Bedford and Kalka \cite{MR0481107}).\\
Therefore, it seems natural to replicate the previously explained bumping method of Bharali and Stens{\o}nes. In this setting, however, it is not clear whether there even exist {\emph{locally defined}} smooth functions, which are positive, homogeneous of degree $2k$ and constant along the (local) leaves of the foliation. This is the content of a question asked by Stens{\o}nes.

The purpose of this paper is to give an {\emph{affirmative answer}} to this question in a slightly more general setting: given a smooth $\mathbb{R}$-homogeneous vector field on an open subset of $\mathbb{C}^2$, such that the collection of complex vector spaces spanned by said vector field is an involutive distribution of real dimension $2$, we construct, assuming that a certain compatibility condition is satisfied, a positive smooth function that is homogeneous of any desired degree and constant along the (local) leaves of the foliation induced by the Frobenius theorem. As a corollary we obtain a positive answer to Stens{\o}nes' question. A precise statement of these results can be found in the following section.

\section{Preliminaries and Statement of Results}

From now on, we fix a point $x\in\mathbb{C}^2\setminus{\{0\}}$, an open neighborhood $N$ of $x$ in $\mathbb{C}^2\setminus\{0\}$ and a vector field
\begin{align*}
\mathcal{V}=
\begin{pmatrix} V_1+i\cdot{V_2} \\ V_3+i\cdot{V_4}\end{pmatrix}
\colon{N}\to\mathbb{C}^2\text{,}
\end{align*}
such that $\mathcal{V}$ vanishes nowhere on $N$.

\theoremstyle{definition}
\newtheorem{topmanfolipaper}[propo]{Assumption}
\begin{topmanfolipaper}
\label{topmanfolipaper}
We make the following additional assumptions:
\begin{enumerate}
\item\label{condi1}{$\mathcal{V}$ is of class $\mathcal{C}^{\infty}$,}
\item\label{condi3}{The collection of $\mathbb{C}$-vector spaces spanned by $\mathcal{V}$ at the points in $N$ is involutive (as a $\mathcal{C}^{\infty}$ distribution of real dimension $2$ on $N$),}
\item\label{condi4}{$\mathcal{V}$ is $\mathbb{R}$-homogeneous of degree $m$ for some positive integer $m$,}
\item\label{condi5}{$p$ and $\mathcal{V}(p)$ are $\mathbb{C}$-linearly independent for all $p\in{N}$.}
\end{enumerate}
\end{topmanfolipaper}

{\noindent
We denote the foliation obtained from applying the Frobenius theorem to the distribution in Property \ref{condi3} as $\mathcal{F}$.
}

\theoremstyle{remark}
\newtheorem{compacondi}[propo]{Remark}
\begin{compacondi}
\label{compacondi}
Properties \ref{condi4} and \ref{condi5} can be interpreted as a compatibility condition between $\mathcal{F}$ and the foliation by spheres centered at the origin.
\end{compacondi}

Under these assumptions, the main result of this paper can be stated as follows:

\theoremstyle{plain}
\newtheorem{maintheoremfoliation}[propo]{Theorem}
\begin{maintheoremfoliation}
\label{maintheoremfoliation}
There exist an open neighborhood $W\subseteq{N}$ of $x$ in $\mathbb{C}^2$ and, given an arbitrary positive integer $n$, a function $g\colon{W}\to\mathbb{R}$ with the following properties:
\begin{itemize}
\item{$g$ is constant along the leaves of the restriction of $\mathcal{F}$ to $W$,}
\item{$g$ is $\mathbb{R}$-homogeneous of degree $n$,}
\item{$g>0$ on $W$,}
\item{$g$ is of class $\mathcal{C}^{\infty}$.}
\end{itemize}
\end{maintheoremfoliation}

\theoremstyle{remark}
\newtheorem{moregeneralgogogo}[propo]{Remark}
\begin{moregeneralgogogo}
\label{moregeneralgogogo}
In order to show the existence of the function $g$ in Theorem \ref{maintheoremfoliation}, one needs to find a function with prescribed behavior with respect to {\emph{both}} the foliation $\mathcal{F}$ and the foliation by spheres centered at the origin. This will be possible because, as mentioned in Remark \ref{compacondi}, the two foliation satisfy a certain ``compatibility condition''.\\
Both the setting in $\mathbb{C}^2$ and the foliation by spherical shells are quite specific. It is likely that one can adjust the method of proof in this paper to derive compatibility conditions that two (or more) foliations need to satisfy in order to admit non-trivial functions with prescribed behavior with respect to each of the foliations, which might be of independent interest.\\
Furthermore, it {\emph{might}} be possible to use an appropriate global version of the implicit function theorem in order to get a global result in the spirit of Theorem \ref{maintheoremfoliation}.
\end{moregeneralgogogo}

We end this section by stating the following corollary, which is relevant for the bumping problem:

\theoremstyle{plain}
\newtheorem{bumppingfoliacoro}[propo]{Corollary}
\begin{bumppingfoliacoro}
\label{bumppingfoliacoro}
Let $x$ be as above and let $P$ be a real-valued polynomial with complex coefficients in two complex variables $(z,w)$ and their conjugates $(\overline{z},\overline{w})$. Assume that
\begin{itemize}
\item{$P$ is $\mathbb{R}$-homogeneous of degree $2k$ for some integer $k\geq{2}$,}
\item{the Complex Hessian matrix of $P$ does not vanish at $x$,}
\item{$x$ does not lie on a complex line through $0\in\mathbb{C}^2$ along which $P$ is harmonic,}
\item{the Levi determinant of $P$ vanishes identically on a neighborhood of $x$ and hence, by real-analyticity, on all of $\mathbb{C}^2$.}
\end{itemize}
Then there exist an open neighborhood $W$ of $x$ in $\mathbb{C}^2$ and a function $g\colon{W}\to\mathbb{R}$ with the following properties:
\begin{itemize}
\item{There exists a smooth foliation of $W$ by complex curves along which $P$ is harmonic,}
\item{$g$ is constant along the leaves of said foliation,}
\item{$g$ is $\mathbb{R}$-homogeneous of degree $2k$,}
\item{$g>0$ on $W$,}
\item{$g$ is of class $\mathcal{C}^{\infty}$.}
\end{itemize}
\end{bumppingfoliacoro}

\section{Proof of Theorem \ref{maintheoremfoliation}}
This section is devoted to proving Theorem \ref{maintheoremfoliation}. We have to find an open neighborhood $W$ of $x$ in $N$ having certain properties. By a slight {\emph{abuse of notation}}, we will (instead of defining the set $W$) shrink the open neighborhood $N$ of $x$ a finite amount of times and establish the existence of a function $g$ with the desired properties on $N$. Each time we shrink $N$, we also restrict the foliation $\mathcal{F}$ accordingly, which we will not always comment on.\\
We begin with the following lemma (as usual, $\mathbb{D}$ denotes the open unit disc centered at $0$ in $\mathbb{C}$):

\theoremstyle{plain}
\newtheorem{localfolichartsubm}[propo]{Lemma}
\begin{localfolichartsubm}
\label{localfolichartsubm}
After shrinking $N$ if necessary and restricting $\mathcal{F}$ accordingly, there exist $0<{\delta}<1$, smooth functions ${u_1},{u_2}\colon{N}\to\mathbb{R}$ and a smooth function $\phi\colon{N}\to\mathbb{C}$, such that:
\begin{enumerate}
\item\label{folipaperannoying1}{the real gradients $\nabla{u_1}$ and $\nabla{u_2}$ are $\mathbb{R}$-linearly independent at every point in $N$ (in particular they vanish nowhere on $N$),}
\item\label{folipaperannoying2}{for $j\in\{1,2\}$, the real gradient $\nabla{u_j}$ is orthogonal to both $\mathcal{V}$ and $i\cdot\mathcal{V}$ at every point in $N$ with respect to the standard inner product on the $\mathbb{R}$-vector space $\mathbb{R}^4$,}
\item\label{folipaperannoying3}{the leaves of $\mathcal{F}$ are precisely the level sets of $u:={u_1}+i{u_2}$, which are complex submanifolds of $\mathbb{C}^2$ of complex dimension $1$,}
\item\label{folipaperannoying4}{the map $\Phi{:=}({\phi},u)$ is a $\mathcal{C}^{\infty}$ diffeomorphism from $N$ onto $\mathbb{D}\times\mathbb{D}$ and ${\Phi}(x)=0$, i.e.\ ${\phi}(x)=0$ and $u(x)=0$,}
\item\label{folipaperannoying5}{For all $t\in{(1-{\delta},1+{\delta})}$, $a,b\in{N}$ we have the following:

If $u(a)=u(b)$ and if $ta,tb\in{N}$, then $u(ta)=u(tb)$.}
\end{enumerate}
\end{localfolichartsubm}

\begin{proof}
By assumption, the collection of $\mathbb{C}$-vector spaces spanned by $\mathcal{V}$ at the points in $N$ is an involutive $\mathcal{C}^{\infty}$ distribution of real dimension $2$ on $N$.

Hence Properties \ref{folipaperannoying1}, \ref{folipaperannoying2}, \ref{folipaperannoying3} and \ref{folipaperannoying4} will follow by applying the Frobenius theorem and the submersion theorem, while shrinking $N$ appropriately several (finitely many) times and restricting the foliation $\mathcal{F}$ accordingly.

Regarding the level sets of $u$ being complex submanifolds of $\mathbb{C}^2$, we note that they are (embedded) smooth submanifolds of $\mathbb{C}^2$ of real dimension $2$, whose tangent spaces at every point can easily be seen to be complex linear subspaces of $\mathbb{C}^2$.

It remains to address Property \ref{folipaperannoying5}. Let $u$ and $\phi$ satisfy Properties \ref{folipaperannoying1}, \ref{folipaperannoying2}, \ref{folipaperannoying3} and \ref{folipaperannoying4}. We make the following claim within the proof ($\mathbb{D}_{1/2}$ denotes the open disc of radius $1/2$ in $\mathbb{C}$ centered at $0$):

\theoremstyle{plain}
\newtheorem*{reductioninlemmabproof}{Claim}
\begin{reductioninlemmabproof}
There exists $0<{\delta}<1$, such that for all $t\in{(1-{\delta},1+{\delta})}$, $a,b\in{\Phi}^{-1}({\mathbb{D}_{1/2}}\times{\mathbb{D}_{1/2}})$ we have the following:
\begin{itemize}
\item{$ta,tb\in{N}$,}
\item{if $u(a)=u(b)$, then $u(ta)=u(tb)$.}
\end{itemize} 
\end{reductioninlemmabproof}

After having shown the claim, we can finish the proof by replacing $N$ by ${\Phi}^{-1}({\mathbb{D}_{1/2}}\times{\mathbb{D}_{1/2}})$ and rescaling $\phi$ and $u$. Hence it suffices to prove the claim.

To this end, we note that ${\Phi}^{-1}({\mathbb{D}_{1/2}}\times{\mathbb{D}_{1/2}})\Subset{N}$, i.e.\ there exists $0<{\delta}<1$, such that $t\cdot{p}\in{N}$, whenever $t\in{(1-{\delta},1+{\delta})}$ and $p\in{{\Phi}^{-1}({\mathbb{D}_{1/2}}\times{\mathbb{D}_{1/2}})}$.

Now let $a,b\in{\Phi}^{-1}({\mathbb{D}_{1/2}}\times{\mathbb{D}_{1/2}})$ with $u(a)=u(b)$. By choice of $\delta$ we have $ta,tb\in{N}$, whenever $t\in{(1-{\delta},1+{\delta})}$. Set $c:=u(a)=u(b)\in\mathbb{D}_{1/2}$ and, for $t\in{(1-{\delta},1+{\delta})}$, define a map
\begin{align*}
\Gamma_t\colon\mathbb{D}_{1/2}\to{N}\text{, }\gamma\mapsto{t\cdot}{\Phi}^{-1}({\gamma},c)\text{,}
\end{align*}
which is welldefined by choice of $\delta$. For $j\in\{1,2\}$, we compute the real gradient of $u_j\circ\Gamma_t\colon\mathbb{D}_{1/2}\to\mathbb{R}$ in real coordinates. We get for $\gamma\in{\mathbb{D}_{1/2}}$:
\begin{align*}
\mathbb{R}^{1\times{2}}\ni\nabla{({u_j\circ\Gamma_t})}({\gamma}) & =\nabla{(u_j)({\Gamma_t}({\gamma}))}\cdot{\operatorname{J}_{\Gamma_t}}({\gamma})\\
& =\nabla{(u_j)(t\cdot{\Gamma_1}({\gamma}))}\cdot{t}\cdot{\operatorname{J}_{\Gamma_1}}({\gamma})\text{,}
\end{align*}
where ${\operatorname{J}_{\Gamma_t}}({\gamma})\in\mathbb{R}^{4\times{2}}$ denotes the Jacobian matrix of $\Gamma_t$ evaluated at $\gamma$. But $u\circ\Gamma_1\equiv{c}$, so $\nabla{(u_j\circ\Gamma_1)}\equiv{0}$ for $j\in\{1,2\}$, i.e.\ we have
\begin{align*}
\nabla{(u_j)({\Gamma_1}({\gamma}))}\cdot{\operatorname{J}_{\Gamma_1}}({\gamma})=0
\end{align*}
for all $\gamma\in\mathbb{D}_{1/2}$, $j\in\{1,2\}$. Hence both columns of ${\operatorname{J}_{\Gamma_1}}({\gamma})$ are orthogonal to $\nabla{(u_j)({\Gamma_1}({\gamma}))}$ with respect to the standard inner product on $\mathbb{R}^4$. Since $\mathcal{V}$ vanishes nowhere on $N$ and since we are in real dimension $4$, we can use Properties \ref{folipaperannoying1} and \ref{folipaperannoying2} to deduce that the columns of ${\operatorname{J}_{\Gamma_1}}({\gamma})$ are contained in the {\emph{complex}} vector space spanned by $\mathcal{V}({\Gamma_1}({\gamma}))$.\\
But, by \ref{condi4} in Assumption \ref{topmanfolipaper}, we immediately get that the columns of ${\operatorname{J}_{\Gamma_t}}({\gamma})=t\cdot{}{\operatorname{J}_{\Gamma_1}}({\gamma})$ are contained in the {\emph{complex}} vector space spanned by $\mathcal{V}(t\cdot{\Gamma_1}({\gamma}))$ and hence orthogonal to $\nabla{(u_j)(t\cdot{\Gamma_1}({\gamma}))}$, $j\in\{1,2\}$, with respect to the standard inner product on $\mathbb{R}^4$. Since $\mathbb{D}_{1/2}$ is connected, this shows together with the above calculation, that $u\circ\Gamma_t$ is constant. Noting that ${\phi}(a),{\phi}(b)\in\mathbb{D}_{1/2}$, we compute:
\begin{align*}
u(ta)=u({{t\cdot}{\Phi}^{-1}({{\phi}(a)},u(a))})=u({{t\cdot}{\Phi}^{-1}({{\phi}(a)},c)})=({u\circ\Gamma_t})({\phi}(a))\text{.}
\end{align*}
Analogously we get that $u(tb)=({u\circ\Gamma_t})({\phi}(b))$. Since $u\circ\Gamma_t$ is constant, we obtain $u(ta)=u(tb)$, as desired.
\end{proof}

\theoremstyle{remark}
\newtheorem{localglobalmiregalleaf}[propo]{Remark}
\begin{localglobalmiregalleaf}
\label{localglobalmiregalleaf}
It should be noted that the distinction between local and global leaves disappears, whenever $N$ is shrunk in a way that it coincides with the open set associated to a foliation chart containing $x$, since we always restrict the foliation appropriately. Because of this, we will not distinguish between local and global leaves of the foliation $\mathcal{F}$ for the remainder of this section, unless stated otherwise.
\end{localglobalmiregalleaf}

\theoremstyle{remark}
\newtheorem{leavesscalehomog}[propo]{Remark}
\begin{leavesscalehomog}
\label{leavesscalehomog}
Property \ref{folipaperannoying5} in Lemma \ref{localfolichartsubm} says that, roughly speaking, the leaves of the foliation scale homogeneously. While this may appear trivial at first glance, it should be noted that this is a property which a priori could easily be destroyed by restricting the foliation to the ``wrong'' open set.
\end{leavesscalehomog}

Armed with Lemma \ref{localfolichartsubm}, we now set for $\tau\in\mathbb{D}$:
\begin{align*}
p_{\tau}:={\Phi}^{-1}(0,{\tau})\in{N}\text{,}
\end{align*}
i.e.\ for each leaf $\{u={\tau}\}$ of $\mathcal{F}$ we pick one point on it, such that this choice depends smoothly on the leaf. Since $N$ is open, we find a ${\delta}_{p_{\tau}}>0$, such that $t\cdot{p_{\tau}}\in{N}$, whenever $1-2{{\delta}_{p_{\tau}}}<{t}<{1}+2{\delta}_{p_{\tau}}$. Hence, for all $\tau\in\mathbb{D}$, we can define a smooth map
\begin{align*}
S_{\tau}\colon{}({{1}-2{\delta}_{p_{\tau}}},{{1}+2{\delta}_{p_{\tau}}})\to\mathbb{D}\text{, }t\mapsto{u(t\cdot{p_{\tau}})}\text{.}
\end{align*}
Intuitively speaking, we go along the real ray through $0\in\mathbb{C}^2$ and $p_{\tau}\in{N}$ and apply $u$, which amounts to checking which leaf we are on. So the derivative $S_{\tau}'$ of $S_{\tau}$ measures ``how the leaf changes along the ray''.

Computing the derivative at $t=1$ in real coordinates, the map
\begin{align*}
\colon\mathbb{D}\to\mathbb{R}^2\text{, }\tau\mapsto{S_{\tau}'}(1)=\begin{pmatrix}
\nabla{u_1}({p_{\tau}}) \\
\nabla{u_2}({p_{\tau}})
\end{pmatrix}\cdot{p_{\tau}}
\end{align*}
(where $p_{\tau}$ is considered as an element of $\mathbb{R}^{4\times{1}}$) defines a smooth vector field on $\mathbb{D}$. If ${S_{\tau}'}(1)$ was to vanish for some $\tau\in\mathbb{D}$, then (analogously to the proof of Property \ref{folipaperannoying5} in Lemma \ref{localfolichartsubm}) that would imply that $p_{\tau}$ was contained in the {\emph{complex}} vector space spanned by $\mathcal{V}({p_{\tau}})$, in contradiction to \ref{condi5} in Assumption \ref{topmanfolipaper}. Hence we have ${S_{\tau}'}(1)\neq{0}$ for all $\tau\in\mathbb{D}$.

Consequently, the collection of {\emph{real}} vector spaces spanned by ${S_{\tau}'}(1)\in\mathbb{R}^2\setminus\{0\}$ at the points $\tau\in\mathbb{D}$ yields a $\mathcal{C}^{\infty}$ distribution of real dimension $1$ on $\mathbb{D}$, which is trivially involutive. The Frobenius theorem implies the following:

\theoremstyle{plain}
\newtheorem{frobbyondiscfolipaper}[propo]{Lemma}
\begin{frobbyondiscfolipaper}
\label{frobbyondiscfolipaper}
There exist an open subset $\Omega$ of $\mathbb{D}$ containing $u(x)=0$ and a $\mathcal{C}^{\infty}$ diffeomorphism
\begin{align*}
\omega{=}({\omega_1},{\omega_2})\colon\Omega\to{(-1,1)\times{(-1,1)}}\text{,}
\end{align*}
such that
\begin{itemize}
\item{the real gradient $\nabla\omega_2$ vanishes nowhere on $\Omega$,}
\item{$\nabla\omega_2{(\tau)}$ and ${S_{\tau}'}(1)$ are orthogonal with respect to the standard inner product on $\mathbb{R}^2$ for all $\tau\in\Omega$.}
\end{itemize}
\end{frobbyondiscfolipaper}

\begin{proof}
This follows from the above considerations.
\end{proof}

If $p_1,p_2\in{\Phi}^{-1}(\mathbb{D}\times{\Omega})$ are points on the same $\mathbb{R}_{\geq{0}}$-ray originating at $0\in\mathbb{C}^2$, then $u({p_1})$ and $u({p_2})$ are not necessarily contained in the same level set of $\omega_2$, since, roughly speaking, one might temporarily leave the set ${\Phi}^{-1}(\mathbb{D}\times{\Omega})$ when going from $p_1$ to $p_2$ along the ray. If, however, we restrict our attention to a suitable smaller open neighborhood of $x$, where this problem does not arise, then the level sets of $\omega_2$ exhibit the desired behavior. That is the content of the following lemma:

\theoremstyle{plain}
\newtheorem{lemmaafolipaper}[propo]{Lemma}
\begin{lemmaafolipaper}
\label{lemmaafolipaper}
There exists an open neighborhood $W_x$ of $x$ in $N$ with the following properties:
\begin{enumerate}
\item\label{lemmaafolipaperprop1}{$W_x\subseteq{{\Phi}^{-1}(\mathbb{D}\times{\Omega})}$ and $W_x\cap{(-W_x)}=\emptyset$,}
\item\label{lemmaafolipaperprop2}{$u(W_x)\Subset\Omega$,}
\item\label{lemmaafolipaperprop3}{there exist an open subset $B_x$ of $\{p\in\mathbb{C}^2\colon{\Vert{p}\Vert{=\Vert{x}\Vert}}\}$ (which is equipped with the subspace topology it inherits from $\mathbb{C}^2$) and a real number ${0<}d_x{<1}$, such that:
\begin{itemize}
\item{$x\in{B_x}$,}
\item{${W_x}=\{t\cdot{p}\in\mathbb{C}^2\colon{{1-d_x}<t<{1+d_x}\text{ and }p\in{B_x}}\}$,}
\item{if $q\in{W_x}$ and if $(1-{d_x})/(1+{d_x})<t<(1+{d_x})/(1-{d_x})$, then $t\cdot{q}\in{N}$ and $u(t\cdot{q})\in\Omega$,}
\end{itemize}
}
\item\label{lemmaafolipaperprop4}{if ${p_1},{p_2}\in{W_x}$ lie on the same $\mathbb{R}_{\geq{0}}$-ray originating at $0\in\mathbb{C}^2$, then ${\omega_2}(u({p_1}))={\omega_2}(u({p_2}))$.}
\end{enumerate}
In fact, whenever $W_x$ is an open neighborhood of $x$ in $N$ having Properties \ref{lemmaafolipaperprop1}, \ref{lemmaafolipaperprop2} and \ref{lemmaafolipaperprop3}, then it will necessarily have Property \ref{lemmaafolipaperprop4}.
\end{lemmaafolipaper}

\begin{proof}
It is clear that there exists an open neighborhood $W_x$ of $x$ in $N$ having Properties \ref{lemmaafolipaperprop1}, \ref{lemmaafolipaperprop2} and \ref{lemmaafolipaperprop3}. We have to show that such a neighborhood necessarily has Property \ref{lemmaafolipaperprop4}.

To this end, let ${p_1},{p_2}\in{W_x}$ lie on the same $\mathbb{R}_{\geq{0}}$-ray originating at $0\in\mathbb{C}^2$. By Property \ref{lemmaafolipaperprop3}, there exist $p\in{B_x}$ and ${t_1},{t_2}\in{({1-d_x},{1+d_x})}$, such that ${p_1}=t_1{p}$ and ${p_2}=t_2{p}$. Hence it suffices to show that the derivative of the (clearly welldefined) smooth map
\begin{align*}
\chi\colon{({1-d_x},{1+d_x})}\to\mathbb{R}\text{, }t\mapsto{\omega_2}(u(t\cdot{p}))
\end{align*}
vanishes identically. So, given $t_0\in{({1-d_x},{1+d_x})}$, we need to show that ${\chi}'({t_0})=0$.

Let $\tau{:=}u(t_0\cdot{p})\in\Omega$ and recall that $p_{\tau}={\Phi}^{-1}(0,{\tau})\in{N}$. We trivially have $u(t_0\cdot{p})=u({p_\tau})$, so, using Property \ref{folipaperannoying5} in Lemma \ref{localfolichartsubm}, we find a $0<\widetilde{\delta}\ll\delta$, such that we have the following for all $t\in{({1-\widetilde{\delta}},{1+\widetilde{\delta}})}$:
\begin{itemize}
\item{$t\cdot{t_0}\cdot{p}$ and $t\cdot{p_{\tau}}$ are contained in $N$,}
\item{$u({t\cdot{t_0}\cdot{p}})=u({t\cdot{p_{\tau}}})$,}
\item{$t\cdot{t_0}\in{({1-d_x},{1+d_x})}$, i.e.\ $t\cdot{t_0}\cdot{p}\in{W_x}$.}
\end{itemize}
Using this, we can define a map
\begin{align*}
\widetilde{\chi}\colon{({1-\widetilde{\delta}},{1+\widetilde{\delta}})}\to\mathbb{R}\text{, }t\mapsto\chi{(t\cdot{t_0})}\text{.}
\end{align*}
Since $t_0\neq{0}$ and $\widetilde{\chi}'(t)={\chi}'(t\cdot{t_0})\cdot{t_0}$ for all $t\in{({1-\widetilde{\delta}},{1+\widetilde{\delta}})}$, it suffices to show that $\widetilde{\chi}'(1)=0$. But, using that $u({t\cdot{t_0}\cdot{p}})=u({t\cdot{p_{\tau}}})$ for all $t\in{({1-\widetilde{\delta}},{1+\widetilde{\delta}})}$ and that $u(p_{\tau})=\tau$, one readily computes
\begin{align*}
\widetilde{\chi}'(1) & =\left(({\nabla\omega_2})(u(t\cdot{p_\tau}))\cdot\begin{pmatrix}
\nabla{u_1}(t\cdot{p_{\tau}}) \\
\nabla{u_2}(t\cdot{p_{\tau}})
\end{pmatrix}\cdot{p_{\tau}}\right)\Bigg\rvert_{t=1}\\
& ={\nabla\omega_2}({\tau})\cdot{{S_{\tau}'}(1)}\\
& =0\text{,}
\end{align*}
where the last equality follows from Lemma \ref{frobbyondiscfolipaper}.
\end{proof}

From now on, we fix an open neighborhood $W_x$ of $x$ as in Lemma \ref{lemmaafolipaper}. Furthermore, we choose an open subset $\widetilde{W_x}$ of $\mathbb{C}^2$ and a $0<\widetilde{\delta_x}\ll{1}$ with the following properties:
\begin{itemize}
\item{$\widetilde{\delta_x}<\delta$ (see Lemma \ref{localfolichartsubm}),}
\item{$x\in\widetilde{W_x}\Subset{W_x}$,}
\item{$t\cdot{q}\in{W_x}$, whenever $q\in\widetilde{W_x}$ and $t\in{(1-{\widetilde{\delta_x}},1+{\widetilde{\delta_x}})}$.}
\end{itemize}

Owing to these properties and Lemma \ref{lemmaafolipaper}, the following maps are welldefined and smooth:
\begin{align*}
\mathcal{M}\colon\widetilde{W_x}\times{(1-{\widetilde{\delta_x}},1+{\widetilde{\delta_x}})} & \to\mathbb{R}\text{,}\\
(q,t) & \mapsto{\omega_1}(u({t\cdot{q}}))-{\omega_1}(u(x))\text{,}\\
\mathcal{N}\colon\widetilde{W_x}\times{(1-{\widetilde{\delta_x}},1+{\widetilde{\delta_x}})} & \to\mathbb{R}\text{,}\\
(q,t) & \mapsto{\omega_2}(u({t\cdot{q}}))-{\omega_2}(u(x))\text{.}
\end{align*}
Noting that $p_{0}={\Phi}^{-1}(0,0)=x$ and using Property \ref{lemmaafolipaperprop4} in Lemma \ref{lemmaafolipaper} we compute:
\begin{align*}
\begin{pmatrix}
\frac{\partial\mathcal{M}}{\partial{t}}(x,1)\\
0
\end{pmatrix} & =\begin{pmatrix}
\frac{\partial\mathcal{M}}{\partial{t}}(x,1)\\[1ex]
\frac{\partial\mathcal{N}}{\partial{t}}(x,1)
\end{pmatrix}\\
& =\left(\operatorname{J}_{\omega}(u(tx))\cdot\begin{pmatrix}
\nabla{u_1}(tx) \\
\nabla{u_2}(tx)
\end{pmatrix}\cdot{x}\right)\Bigg\rvert_{t=1}\\
& =\operatorname{J}_{\omega}(u(x))\cdot\begin{pmatrix}
\nabla{u_1}({p_{0}}) \\
\nabla{u_2}({p_{0}})
\end{pmatrix}\cdot{p_{0}}\\
& =\operatorname{J}_{\omega}(u(x))\cdot{{S_{0}'}(1)}\text{.}
\end{align*}
But we have ${S_{\tau}'}(1)\neq{0}$ for all $\tau\in\mathbb{D}$ and $\omega$ is a $\mathcal{C}^{\infty}$ diffeomorphism, so we can conclude that
\begin{align*}
\frac{\partial\mathcal{M}}{\partial{t}}(x,1)\neq{0}\text{.}
\end{align*}
So, since $\mathcal{M}$ is smooth and we clearly have $\mathcal{M}(x,1)=0$, the implicit function theorem implies the following lemma:

\theoremstyle{plain}
\newtheorem{defofcalitfolipaper}[propo]{Lemma}
\begin{defofcalitfolipaper}
\label{defofcalitfolipaper}
There exist an open neighborhood $V_x$ of $x$ in $\widetilde{W_x}$, an open neighborhood $\mathcal{I}$ of $1$ in ${(1-{\widetilde{\delta_x}},1+{\widetilde{\delta_x}})}$ and a smooth map $\mathcal{T}\colon{V_x}\to\mathcal{I}$ with $\mathcal{T}(x)=1$, such that for all $(q,t)\in{V_x}\times\mathcal{I}$ we have:
\begin{align*}
\mathcal{M}(q,t)=0\text{ if and only if }t=\mathcal{T}(q)\text{.}
\end{align*}
\end{defofcalitfolipaper}

\begin{proof}
This follows from the above considerations.
\end{proof}

Pick an open subset $\widetilde{V_x}$ of $\mathbb{C}^2$, an open subset $\widetilde{B_x}$ of $\{p\in\mathbb{C}^2\colon{\Vert{p}\Vert{=\Vert{x}\Vert}}\}$ and a $0<{\lambda_x}\ll{1}$, such that:
\begin{itemize}
\item{$x\in\widetilde{V_x}\Subset{V_x}$ and $\widetilde{V_x}\cap{(-\widetilde{V_x})}=\emptyset$ and $x\in\widetilde{B_x}$,}
\item{$\widetilde{V_x}=\{t\cdot{p}\in\mathbb{C}^2\colon{{1-\lambda_x}<t<{1+\lambda_x}\text{ and }p\in\widetilde{B_x}}\}$.}
\end{itemize}
Let $n$ be a positive integer, as in the statement of Theorem \ref{maintheoremfoliation}. We now define:
\begin{align*}
g\colon\widetilde{V_x}\to\mathbb{R}\text{, }q\mapsto{\left({\frac{1}{\mathcal{T}(q)}}\right)}^n\text{,}
\end{align*}
which is clearly welldefined. Since we can shrink $N$, it suffices to show that $g$ has the desired properties on $\widetilde{V_x}$.

It is obvious that $g$ is of class $\mathcal{C}^{\infty}$ and everywhere $>0$. Now assume that $q_1,q_2\in\widetilde{V_x}$ lie on the same leaf of the restriction of $\mathcal{F}$ to $\widetilde{V_x}$. In particular we have $u({q_1})=u({q_2})$. We have to show that $g({q_1})=g({q_2})$; so it suffices to prove that $\mathcal{T}({q_1})=\mathcal{T}({q_2})$.\\
Owing to the choices we made, we have $\mathcal{T}({q_2})\in{(1-{\delta},1+{\delta})}$ and the points $q_1$, $q_2$, $\mathcal{T}({q_2})\cdot{q_1}$ and $\mathcal{T}({q_2})\cdot{q_2}$ are contained in $N$. Since $u({q_1})=u({q_2})$, we can hence apply Lemma \ref{localfolichartsubm} to obtain $u({\mathcal{T}({q_2})\cdot{q_1}})=u({\mathcal{T}({q_2})\cdot{q_2}})$. Since $({q_1},\mathcal{T}({q_2}))$ and $({q_2},\mathcal{T}({q_2}))$ are contained in $V_x\times\mathcal{I}$, we get
\begin{align*}
\mathcal{M}({q_1},\mathcal{T}({q_2}))=\mathcal{M}({q_2},\mathcal{T}({q_2}))=0\text{;}
\end{align*}
Lemma \ref{defofcalitfolipaper} then implies that $\mathcal{T}({q_1})=\mathcal{T}({q_2})$, as desired.

It remains to show that $g$ is $\mathbb{R}$-homogeneous of degree $n$. To this end, let $q\in\widetilde{V_x}$, $t\in\mathbb{R}$ and assume $t\cdot{q}\in\widetilde{V_x}$. We have to show that $g(tq)=t^n\cdot{g(q)}$. By choice of $\widetilde{V_x}$ one readily reduces to the case that $q\in\widetilde{B_x}$ and $t\in{(1-{\lambda_x},1+{\lambda_x})}$.

The map
\begin{align*}
\colon{(1-{\lambda_x},1+{\lambda_x})} & \to\mathbb{R}^2\text{,}\\
s & \mapsto{(\mathcal{M}{(sq,\mathcal{T}(sq))},\mathcal{N}{(sq,\mathcal{T}(sq))})}
\end{align*}
is welldefined and constant by Lemmas \ref{defofcalitfolipaper} and \ref{lemmaafolipaper} and by the choices we made. Since $\omega$ is a $\mathcal{C}^{\infty}$ diffeomorphism, that implies that the following map is welldefined and constant:
\begin{align*}
\colon{(1-{\lambda_x},1+{\lambda_x})}\to\mathbb{C}\text{, }
s\mapsto{u(\mathcal{T}(sq)\cdot{sq})}\text{.}
\end{align*}
By differentiating and thereupon using Property \ref{condi5} in Assumption \ref{topmanfolipaper} analogously to above, we obtain:
\begin{align*}
\mathcal{T}(sq)+s\cdot\nabla\mathcal{T}(sq)\cdot{q}=0\text{ for all }s\in{(1-{\lambda_x},1+{\lambda_x})}\text{,}
\end{align*}
which directly implies that the map
\begin{align*}
\colon{(1-{\lambda_x},1+{\lambda_x})}\to\mathbb{R}\text{, }
s\mapsto{\mathcal{T}(sq)\cdot{s}}
\end{align*}
is constant. Since $t\in{(1-{\lambda_x},1+{\lambda_x})}$, we get $\mathcal{T}(q)=\mathcal{T}(tq)\cdot{t}$. A straightforward calculation then gives $g(tq)=t^n\cdot{g(q)}$, as desired.

\section{Proof of Corollary \ref{bumppingfoliacoro}}
This section is devoted to proving Corollary \ref{bumppingfoliacoro}. To this end, let $P$ and $x$ be as in the statement of Corollary \ref{bumppingfoliacoro}.

\theoremstyle{definition}
\newtheorem{xxcomplehessiannotationx2}[propo]{Notation}
\begin{xxcomplehessiannotationx2}
\label{xxcomplehessiannotationx2}
We denote the Complex Hessian Matrix or the Levi Matrix of $P$ as $H_P$, i.e.\
\begin{align*}
\arraycolsep=0.2pt\def\arraystretch{1.4}
H_P=\left( \begin{array}{ccc}
\frac{\partial^2 P}{\partial{z}\partial{\overline{z}}} & \frac{\partial^2 P}{\partial{w}\partial{\overline{z}}} \\
\frac{\partial^2 P}{\partial{z}\partial{\overline{w}}} & \frac{\partial^2 P}{\partial{w}\partial{\overline{w}}} \end{array} \right)\text{.}
\end{align*}
% As customary, we call $\det{H_P}$ the {\emph{Levi determinant}} of $P$.
\end{xxcomplehessiannotationx2}

Firstly we note that, due to the formulation of Corollary \ref{bumppingfoliacoro} and the proof of Theorem \ref{maintheoremfoliation}, we neither have to concern ourselves with the difference between local leaves and global leaves nor with the smoothness of the foliation.

Furthermore, after having checked the assumptions for applying Theorem \ref{maintheoremfoliation}, it will be immediate that, after shrinking $W$ if necessary, the leaves are submanifolds of $\mathbb{C}^2$ of real dimension $2$ and hence complex curves, since their tangent spaces at every point are {\emph{complex}} linear subspaces of $\mathbb{C}^2$. We also remark that harmonicity of $P$ along the leaves of the foliation will be a welldefined notion due to the leaves being complex curves.

Because of these remarks we will simply check that the assumptions for applying Theorem \ref{maintheoremfoliation} are satisfied.

By assumption we have $H_P{(x)}\neq{0}$, so that at least one of the two vector fields
\begin{align*}
(z,w)\mapsto\begin{pmatrix} -\frac{\partial^2 P}{\partial{w}\partial{\overline{z}}} \\[1ex] \phantom{-}\frac{\partial^2 P}{\partial{z}\partial{\overline{z}}}\end{pmatrix}(z,w)\text{ and }(z,w)\mapsto\begin{pmatrix} -\frac{\partial^2 P}{\partial{w}\partial{\overline{w}}} \\[1ex] \phantom{-}\frac{\partial^2 P}{\partial{z}\partial{\overline{w}}}\end{pmatrix}(z,w)
\end{align*}
does not vanish in $x$. Let $\mathcal{V}$ be one of these two vector fields, such that $\mathcal{V}(x)\neq{0}$. Since the Levi determinant of $P$ vanishes on $\mathbb{C}^2$, we get $H_P\cdot\mathcal{V}\equiv{0}$ by choice of $\mathcal{V}$, which shows that $P$ is indeed harmonic along the leaves of the foliation whose existence we are about to establish.\\
Now we pick an open neighborhood $N\subseteq\mathbb{C}^2\setminus\{0\}$ of $x$, such that $\mathcal{V}$ vanishes nowhere on $N$ and such that $N$ does not meet a complex line through $0$ along which $P$ is harmonic. The latter is possible by assumption on $P$ and $x$.

It remains to verify the properties in Assumption \ref{topmanfolipaper}. Properties \ref{condi1} and \ref{condi4} are clear. Noting that $\mathcal{V}$ does not vanish on $N$ and hence defines a $\mathcal{C}^{\infty}$ distribution of real dimension $2$ on $N$, Property \ref{condi3} follows from directly computing the Lie bracket $[\mathcal{V},i\cdot\mathcal{V}]$ in real Cartesian coordinates and making use of the fact that $\det{H_P}\equiv{0}$ on $\mathbb{C}^2$.\\
By restricting the foliation on $N$ obtained from Frobenius theorem to a foliation chart containing $x$ and subsequently replacing $N$ by the open set associated to said chart, we can assume that the leaves of the foliation are (embedded) complex submanifolds of $N$. The leaves of the restricted foliation are precisely the plaques of the original foliation in the foliation chart of consideration.

Using the assumptions on $P$ and $x$ and the properties established thus far, we can (by an argument similar to the one appearing in the previous section) find smooth functions $u={u_1}+i{u_2}\colon{N}\to\mathbb{D}$ and $\phi\colon{N}\to\mathbb{D}$ having Properties \ref{folipaperannoying1}, \ref{folipaperannoying2}, \ref{folipaperannoying3} and \ref{folipaperannoying4} from Lemma \ref{localfolichartsubm} (it should be noted that this is potentially accompanied by shrinking $N$  and restricting the foliation yet again).

In order to verify Property \ref{condi5} in Assumption \ref{topmanfolipaper}, we assume for the sake of a contradiction that there exists a point $p\in{N}$, such that $p$ and $\mathcal{V}(p)$ are $\mathbb{C}$-linearly dependent. Since $\mathcal{V}(p)\neq{0}$, we have that $p$ in contained in the complex vector subspace of $\mathbb{C}^2$ spanned by $\mathcal{V}(p)$. Since $P$ is homogeneous, we find a small $0<{\delta}\ll{1}$ (not to be confused with the $\delta$ appearing in the previous section), such that for all $t\in\left({1-{\delta},1+{\delta}}\right)$ we have:
\begin{align*}
t\cdot{p}\in{N}\text{ and }{p}\in\operatorname{Span}_{\mathbb{C}}\left(\left\{\mathcal{V}(t\cdot{p})\right\}\right)\text{.}
\end{align*}
We consider the map
\begin{align*}
S\colon{\left({1-{\delta},1+{\delta}}\right)}\to\mathbb{D}\text{, }t\mapsto{u(t\cdot{p})}\text{.}
\end{align*}
In real coordinates, the derivative at $t\in\left({1-{\delta},1+{\delta}}\right)$ computes to
\begin{align*}
S'(t)=\begin{pmatrix}
\nabla{u_1}(t\cdot{p}) \\
\nabla{u_2}(t\cdot{p})
\end{pmatrix}\cdot{p}\text{,}
\end{align*}
where $p$ is considered as an element of $\mathbb{R}^{4\times{1}}$ and the gradients are considered as elements of $\mathbb{R}^{1\times{4}}$. As seen above, $p$ is contained in the $\mathbb{R}$-vector space spanned by $\mathcal{V}(t\cdot{p})$ and $i\cdot\mathcal{V}(t\cdot{p})$, which implies $S'(t)=0$ by the defining properties of $u$. It follows that $S$ is constant, so the points $t\cdot{p}$, $t\in\left({1-{\delta},1+{\delta}}\right)$, are all contained in $L:=\{q\in{N}\colon{u(q)=u(p)}\}$. Recalling the behavior of the level sets of $u$, we find an open neighborhood $\mathcal{Z}$ of $p$ in $N$ and a holomorphic coordinate system $({\zeta}_1,{\zeta}_2)\colon\mathcal{Z}\to\mathbb{C}^2$, such that
\begin{align*}
\{q\in\mathcal{Z}\colon{\zeta_2}(q)=0\}=\mathcal{Z}\cap{L}\text{.}
\end{align*}
Let $X\subseteq\mathbb{C}$ be a small open disc centered at $1$, such that $s\cdot{p}\in\mathcal{Z}$ for all $s\in{X}$. The map
\begin{align*}
\colon{X}\to\mathbb{C}\text{, }s\mapsto{\zeta_2}(s\cdot{p})
\end{align*}
is holomorphic and vanishes on $X\cap{\left({1-{\delta},1+{\delta}}\right)}$; hence said map vanishes on all of $X$. But this immediately gives that
\begin{align*}
(\colon{X}\to\mathbb{C}\text{, }s\mapsto{u(s\cdot{p})})\equiv{u(p)}\text{.}
\end{align*}
Writing $s=a+ib$, applying $\partial{}/\partial{a}$ and considering $p$ as an element of $\mathbb{R}^{4\times{1}}$ again, we get the following in real coordinates:
\begin{align*}
\begin{pmatrix}
\nabla{u_1}(s\cdot{p}) \\
\nabla{u_2}(s\cdot{p})
\end{pmatrix}\cdot{p}=0\text{ for all }s\in{X}\text{.}
\end{align*}
So, with respect to the standard inner product on $\mathbb{R}^4$, we get that $p$ is contained in the orthogonal complement of the $\mathbb{R}$-span of $\nabla{u_1}(s\cdot{p})$ and $\nabla{u_2}(s\cdot{p})$, for all $s\in{X}$. But $\nabla{u_1}(s\cdot{p})$ and $\nabla{u_2}(s\cdot{p})$ are linearly independent over $\mathbb{R}$, i.e.\ said orthogonal complement has real dimension $2$ and hence equals $\operatorname{Span}_{\mathbb{C}}\left(\left\{\mathcal{V}(s\cdot{p})\right\}\right)$. We get
\begin{align*}
{p}\in\operatorname{Span}_{\mathbb{C}}\left(\left\{\mathcal{V}(s\cdot{p})\right\}\right)\text{ for all }s\in{X}\text{.}
\end{align*}
Since $H_P\cdot\mathcal{V}\equiv{0}$, this implies that $H_P{(s\cdot{p})}\cdot{p}=0$ in complex coordinates for all $s\in{X}$. Noting that this expression is real-analytic in $s\in\mathbb{C}$, we deduce that $H_P{(s\cdot{p})}\cdot{p}=0$ for all $s\in\mathbb{C}$. But that implies that $P$ is harmonic along the complex line through $0$ and $p$. Since, however, $N$ was chosen to not meet a complex line through $0$ along which $P$ is harmonic, we get $p\not\in{N}$. We have arrived at the desired contradiction.

\bibliographystyle{amsplain}
\bibliography{refspaper3}

\end{document}